\documentclass[11pt]{article}
\usepackage{amsmath,amsfonts,amsthm,amssymb, mathtools}
\usepackage{mathrsfs, graphicx,color,latexsym, tikz, calc,
}
\usepackage{tikz-cd}
\usepackage{graphicx}
\usepackage{epstopdf}
\usepackage{enumerate}
\usepackage{caption}
\usepackage{stmaryrd}
\usepackage{xcolor}
\usepackage{hyperref}
\hypersetup{
    colorlinks=true,
    citecolor=blue,
    linkcolor=blue,
    urlcolor=blue
}
\usepackage{cite}

\usetikzlibrary{shadows}
\usetikzlibrary{patterns,arrows,decorations.pathreplacing}

\usepackage[margin=2.5cm]{geometry}
\usepackage{ulem}

\newtheorem{theorem}{Theorem}[section]

\newtheorem{lemma}[theorem]{Lemma}

\newtheorem{prop}[theorem]{Proposition}

\allowdisplaybreaks[4]

\title{\bf \Large }
\date{\today }
\title{{\bf \Large 
 Random partition for Tokushige's $r$-wise intersecting conjecture}\footnote{ Lihua Feng was supported by the NSFC (Nos. 12271527 and 12471022) and NSF of Qinghai Province (No. 2025-ZJ-902T). E-mail addresses: \url{wuyjmath@163.com} (Y. Wu), \url{fenglh@163.com} (L. Feng). }
\author{
{\small  Yongjiang Wu,\ \ Lihua Feng\footnote{Corresponding author}
}\\[2mm]
\small School of Mathematics and Statistics, HNP-LAMA, Central South University\\
 \small Changsha, Hunan, 410083, China\\ 
}}

\begin{document}
\maketitle
\begin{abstract}   
Let $r\ge 3$ and let $1>p_1\ge p_2\ge\cdots\ge p_n>0$.  
Let $\mu_{\mathbf p}$ denote the product measure on $2^{[n]}$ where each coordinate $i$ is included independently with probability $p_i$. A family $\mathcal A\subseteq 2^{[n]}$ is  $r$-wise intersecting if $A_1\cap\cdots\cap A_r\neq\emptyset$ for all $A_1,\ldots,A_r\in\mathcal A$. 
In 2022, Tokushige proved that if
$p_2<\frac{r-1}{r}$, then every $r$-wise intersecting family
$\mathcal{A}\subseteq 2^{[n]}$ satisfies $\mu_{\mathbf p}(\mathcal{A})\le p_1$, with equality only for
stars centred at coordinates of maximum probability.  He conjectured that the
hypothesis $p_2<\frac{r-1}{r}$ can be replaced by $p_{r+1}<\frac{r-1}{r}$.  In this paper, we prove this
conjecture in full.  The key novelty is the introduction of a new random partition  method, which reduces the problem to at most $r$ coordinates and solves it exactly, thereby fully covering all cases with multiple supercritical coordinates.
\end{abstract}



\section{Introduction}
For a probability vector $\mathbf{p}=(p_1,\ldots,p_n)\in(0,1)^n$ and a family $\mathcal{A}\subseteq 2^{[n]}$, define the \textit{$\mathbf{p}$-biased measure}
$$
\mu_{\mathbf{p}}(\mathcal{A})=\sum_{A\in\mathcal{A}}\prod_{i\in A}p_i\prod_{j\notin A}(1-p_j).
$$
When $p_1=\cdots=p_n=p$, this is the usual \textit{$p$-biased measure}, denoted by $\mu_p$.
A family $\mathcal{A}\subseteq 2^{[n]}$ is called \textit{$r$-wise $t$-intersecting} if
$$
|A_1\cap\cdots\cap A_r|\geq t\qquad\text{for all }A_1,\ldots,A_r\in\mathcal{A}.
$$
When $t=1$, this is the usual \textit{$r$-wise intersecting family}; when $r=2$, it is the usual \textit{$t$-intersecting family}; when $r=2$ and $t=1$, it is simply \textit{intersecting}. 
The \textit{star} centred at $i$ is
$$
 \mathcal{S}_i=\{A\subseteq[n]:i\in A\}.
$$ 
It is $r$-wise intersecting and has measure $\mu_{\mathbf{p}}(\mathcal{S}_i)=p_i$.

The $p$-biased measure formulation of intersection problems has been extensively studied. For the uniform bias $0<p<1$, Friedgut \cite{F08} developed a spectral approach, showing that for $p\le \frac{1}{t+1}$, every $t$-intersecting family satisfies $\mu_p(\mathcal{A})\le p^t$. The non-uniform version was introduced by Fishburn, Frankl, Freed, Lagarias and Odlyzko \cite{F86}, who proved that if $p_1\ge\cdots\ge p_n$ and $p_2\le \frac{1}{2}$, then every intersecting family satisfies $\mu_{\mathbf{p}}(\mathcal{A})\le p_1$. Suda, Tanaka and Tokushige \cite{S17} later conjectured that the condition $p_2\le \frac{1}{2}$ can be relaxed to $p_3\le \frac{1}{2}$. Tokushige \cite{T22} proved this conjecture under the additional assumptions $p_1\le 1/2$ or $1-p_2>p_3$, using the high-dimensional Hoffman bound of Filmus, Golubev and Lifshitz \cite{F21}. Very recently, Wu and Feng \cite{W26} confirmed the full conjecture using the generating set method, and moreover established a natural $t$-intersecting generalization.

For $r$-wise intersecting families with $r\ge 3$, Frankl and Tokushige \cite{F03} studied the uniform bias case, proving that if $p\leq\frac{r-1}{r}$, then every $r$-wise intersecting family satisfies $\mu_p(\mathcal A)\le p$. Filmus, Golubev and Lifshitz \cite{F21} later gave an alternative proof of this result using the high-dimensional Hoffman bound for hypergraphs. The related $r$-cross intersecting setting was also studied recently by Chang, Liu and Liu \cite{C26}. Tokushige \cite{T22} extended the Frankl-Tokushige theorem to the non-uniform setting.

\begin{theorem}[Tokushige \cite{T22}]\label{T22}
Let $r\ge 3$ and let $1>p_1\ge\cdots\ge p_n>0$. If $p_2<\frac{r-1}{r}$ and $\mathcal A\subseteq 2^{[n]}$ is $r$-wise intersecting, then
$$
\mu_{\mathbf p}(\mathcal A)\le p_1.
$$
Moreover, equality holds if and only if $\mathcal A=\mathcal S_i$ for some $i\in[n]$ with $p_i=p_1$.
\end{theorem}

Tokushige further conjectured that the hypothesis $p_2<\frac{r-1}{r}$ can be replaced by $p_{r+1}<\frac{r-1}{r}$. In this paper, we prove this conjecture in full.

\begin{theorem}\label{M1}
Let $r\ge3$ and let $1>p_1\ge\cdots\ge p_n>0$, and use the convention $p_i=0$ for $i>n$. Suppose that
$
p_{r+1}<\frac{r-1}{r}.
$
If $\mathcal{A}\subseteq2^{[n]}$ is $r$-wise intersecting, then
$$
\mu_{\mathbf{p}}(\mathcal{A})\le p_1.
$$
Moreover, equality holds if and only if $\mathcal{A}=\mathcal{S}_i$ for some $i\in[n]$ with $p_i=p_1$.
\end{theorem}

The new difficulty is the case in which several coordinates have probability larger than $\frac{r-1}{r}$. 
Our proof introduces a random partition method that solves the resulting finite linear program on at most $r$ coordinates. The proof has two main components. We first replace the given family by its upset, allowing us to assume monotonicity and to raise all supercritical coordinates to the common maximum probability without decreasing the measure. We then apply a critical coupling, in the spirit of Chang, Liu and Liu \cite{C26}, to absorb all subcritical coordinates, namely those whose probabilities are at most $\frac{r-1}{r}$, into a single additive bound. This leaves only the supercritical coordinates, of which there are at most $r$ by the hypothesis $p_{r+1}<\frac{r-1}{r}$. The resulting problem is a small linear program on these remaining coordinates, which we solve explicitly via a random partition argument. This gives a clean and self-contained proof of the extremal bound. The equality case requires additional work: we show that equality forces the top slice to be the whole tail cube, which in turn implies that the supports of the remaining slices form an $(r-1)$-wise intersecting family. This extra set-theoretic information collapses the equality face to a single common coordinate.

This paper is organized as follows. Section \ref{se2} collects the preliminary tools, including the upset reduction and the critical additive bound. Section \ref{se3} introduces the slicing argument that reduces the problem to a finite linear program on the supercritical coordinates, and proves the key supercritical slice inequality using a random partition argument. Section  \ref{se4} presents the collapse lemma and completes the proof of the main theorem.

\section{Basic tools}\label{se2}

 A family $\mathcal{A}\subseteq 2^{[n]}$ is called \textit{increasing} if it is closed under taking supersets: whenever $A\in\mathcal A$ and $A\subseteq B$, we have $B\in\mathcal{A}$. The \textit{upset} $\mathcal A^\uparrow$ of $\mathcal{A}$ is  defined by
$$
\mathcal{A}^\uparrow=\{B\subseteq[n]:\text{there exists }A\in\mathcal{A}\text{ such that }A\subseteq B\}.
$$
We shall use the following standard facts about increasing families.

\begin{lemma}\label{lem:monotone}
Let $\mathcal{A}\subseteq2^{[n]}$ be $r$-wise intersecting.
Then $\mathcal{A}^\uparrow$ is $r$-wise intersecting and
$\mu_{\mathbf p}(\mathcal{A}^\uparrow)\ge\mu_{\mathbf p}(\mathcal{A})$.
If $\mathcal{A}$ is increasing and a coordinate probability $p_j$ is increased, then
$\mu_{\mathbf p}(\mathcal{A})$ cannot decrease.  More precisely, writing
\[
 \mathcal{A}_0=\{A\subseteq[n]\setminus\{j\}:A\in\mathcal{A}\},\qquad
 \mathcal{A}_1=\{A\subseteq[n]\setminus\{j\}:A\cup\{j\}\in\mathcal{A}\},
\]
we have $\mathcal{A}_0\subseteq\mathcal{A}_1$, and the measure is strictly increasing in
$p_j$ unless $\mathcal{A}_0=\mathcal{A}_1$.
\end{lemma}

\begin{proof}
If $B_1,\ldots,B_r\in\mathcal{A}^\uparrow$, choose $A_i\in\mathcal{A}$ with
$A_i\subseteq B_i$.  Then
$$
 \emptyset\ne A_1\cap\cdots\cap A_r\subseteq B_1\cap\cdots\cap B_r.
$$
Hence, $\mathcal{A}^\uparrow$ is $r$-wise intersecting.  For the measure inequality, observe that $\mathcal A\subseteq\mathcal A^{\uparrow}$. Since every atom of the product measure has positive weight, adding sets to a family can only increase its measure. It follows that $\mu_{\mathbf p}(\mathcal A^{\uparrow})\ge\mu_{\mathbf p}(\mathcal A)$.

For monotonicity, condition on all coordinates except $j$.  If
$\mathbf p_{-j}$ denotes the remaining probability vector, then
\[
 \mu_{\mathbf p}(\mathcal{A})=(1-p_j)\mu_{\mathbf p_{-j}}(\mathcal{A}_0)
        +p_j\mu_{\mathbf p_{-j}}(\mathcal{A}_1).
\]
Since $\mathcal{A}$ is increasing, $\mathcal{A}_0\subseteq\mathcal{A}_1$.  Thus, the derivative
with respect to $p_j$ is
$\mu_{\mathbf p_{-j}}(\mathcal{A}_1\setminus\mathcal{A}_0)\ge0$, and it is zero only when
$\mathcal{A}_1=\mathcal{A}_0$, because all atoms have positive measure.
\end{proof}

Families $\mathcal F_1,\ldots,\mathcal F_r\subseteq 2^{[N]}$ are \textit{$r$-cross intersecting} if $F_1\cap\cdots\cap F_r\neq\emptyset$ for all $F_i\in\mathcal F_i$. The following lemma is the key additive estimate for such families.
The next lemma is the critical additive estimate which absorbs all coordinates
whose probabilities are at most $\frac{r-1}{r}$.

\begin{lemma}[Critical additive bound]\label{lem:additive}
Let $r\ge2$ and let $\mathbf p=(p_1,\ldots,p_N)$ satisfy
$0\le p_j\le\frac{r-1}{r}$ for every $j$. If
$\mathcal{F}_1,\ldots,\mathcal{F}_r\subseteq2^{[N]}$ are $r$-cross intersecting, then
$$
 \sum_{i=1}^r\mu_{\mathbf p}(\mathcal{F}_i)\le r-1.
$$
\end{lemma}
\begin{proof}

For each coordinate $j\in[N]$,  independently choose a random proper subset
$S_j\subsetneq[r]$ such that
$$
 \mathbb P(i\in S_j)=p_j,\qquad\text{for every }i\in[r].
$$
To construct such $S_j$, first choose a random integer $K_j\in\{0,1,\ldots,r-1\}$ with expectation
$
\mathbb E[K_j]=r p_j.
$
This is possible because $0\le r p_j\le r-1$, so the desired expectation lies in the interval $[0,r-1]$. Condition on  $K_j=k$, choose $S_j$ uniformly from all $k$-element subsets of $[r]$. Then for each fixed $i\in[r]$,
$$
\mathbb P(i\in S_j)
=\sum_{k=0}^{r-1}\mathbb P(K_j=k)\cdot \mathbb P(i\in S_j\mid K_j=k)
=\sum_{k=0}^{r-1}\mathbb P(K_j=k)\cdot \frac{k}{r}
=\frac{\mathbb E[K_j]}{r}
=p_j.
$$

Now define random sets
$$
X_i=\{j\in[N]: i\in S_j\},\qquad i=1,\ldots,r.
$$
For each fixed $i$, the coordinate $j$ belongs to $X_i$ with probability $p_j$, independently across different $j$. Hence, $X_i$ has distribution $\mu_{\mathbf p}$, and therefore
$
\mathbb P(X_i\in\mathcal F_i)=\mu_{\mathbf p}(\mathcal F_i)
$
for each $i\in[r]$.
We claim that the events $\{X_i\in\mathcal F_i\}$ cannot all occur simultaneously. Indeed, for every coordinate $j\in[N]$, the set $S_j$ was chosen to be a proper subset of $[r]$, so there exists at least one index $i\in[r]$ such that $i\notin S_j$. Equivalently, no coordinate $j$ belongs to all of the sets $X_1,\ldots,X_r$. It follows that
$$
X_1\cap X_2\cap\cdots\cap X_r=\emptyset.
$$
 If all events $X_i\in\mathcal F_i$ occurred simultaneously, then the intersection $\bigcap_{i=1}^r X_i$ would be nonempty, since the $\mathcal F_i$ are $r$-cross intersecting. This contradicts the fact that the intersection is always empty. Therefore, 
$$
\sum_{i=1}^r \mathbf 1_{\{X_i\in\mathcal F_i\}}\le r-1,
$$
where $\mathbf 1_{\{X_i\in\mathcal F_i\}}$ is the indicator that $X_i\in\mathcal F_i$.
Taking expectations on both sides gives
$$
\sum_{i=1}^r \mu_{\mathbf p}(\mathcal F_i)=\sum_{i=1}^r \mathbb P(X_i\in\mathcal F_i)\le r-1.
$$
This completes the proof.
\end{proof}

\section{Slicing above the critical threshold}\label{se3}

We now reduce the problem to a small number of coordinates.  Put $$p_*=\frac{r-1}{r},\qquad s=r-1.$$  Let $H\subseteq[n]$ be a set of high coordinates, $|H|=m\le r$,
and let $T=[n]\setminus H$ be the tail.  In the applications, $H$ will be
$\{i:p_i>p_*\}$, and all tail probabilities will be at most $p_*$.

Assume for this section that all coordinates in $H$ have a common probability
$p>p_*$, and write $$q=1-p,\qquad t=\frac{q}{p}.$$  Then $0<t<\frac{1}{s}$.  Let
$\mathcal{A}\subseteq2^{H\cup T}$ be increasing and $r$-wise intersecting.  For
$B\subseteq H$, define the tail section
$$
 \mathcal{A}_B=\{C\subseteq T:B\cup C\in\mathcal{A}\},\qquad
 w_B=\mu_T(\mathcal{A}_B)=\sum_{C\in\mathcal A_B}\prod_{i\in C}p_i\prod_{i\in T\setminus C}(1-p_i),
$$
where $\mu_T$ is the product measure on the tail $T$. We write $\mu$ for the product measure on $2^{H\cup T}$ with respect to the probability vector that assigns probability $p$ to each coordinate in $H$ and probability $p_i$ to each coordinate $i\in T$. 
 Then
\begin{equation}\label{eq:slice-measure}
 \mu(\mathcal{A})=\sum_{B\subseteq H}p^{|B|}q^{m-|B|}w_B.
\end{equation}

\begin{lemma}\label{lem:slice}
If $B_1,\ldots,B_r\subseteq H$ and
$B_1\cap\cdots\cap B_r=\emptyset$, then
$$
 w_{B_1}+\cdots+w_{B_r}\le r-1.
$$
\end{lemma}
\begin{proof}
For every $C_i\in\mathcal{A}_{B_i}$, the set $B_i\cup C_i$ belong to $\mathcal{A}$.  Since
$\mathcal{A}$ is $r$-wise intersecting and $B_1\cap\cdots\cap B_r=\emptyset$, the
common point of $B_i\cup C_i$ must lie in the tail $T$.  Hence,
$C_1\cap\cdots\cap C_r\ne\emptyset$.  Therefore,
$\mathcal{A}_{B_1},\ldots,\mathcal{A}_{B_r}$ are $r$-cross-intersecting on the tail.
Since the tail probabilities are all at most $p_*$,  Lemma \ref{lem:additive} gives
the desired inequality.
\end{proof}

We shall use a small random partition identity.

\begin{lemma}[Partition law]\label{lem:partition}
Let $S$ be a finite set with $1\leq |S|=k\le s$ and let $0<t\le\frac{1}{s}$.  There is a
probability distribution on set partitions $\Pi$ of $S$ (with nonempty blocks only) such that, for every
nonempty $D\subseteq S$,
$$
 \mathbb P(D\in\Pi)=t^{|D|-1}(1-t)^{k-|D|}.
$$
\end{lemma}
\begin{proof}

For a partition $\pi$ of $S$, let $|\pi|$ denote its number of blocks. Define
$$
\mathbb P(\Pi=\pi)=t^{k-|\pi|}\prod_{j=1}^{|\pi|-1}(1-jt). 
$$
The factors are nonnegative because $|\pi|-1\le k-1\le s-1$ and
$t\le \frac{1}{s}$.

We first verify normalization. Let $\genfrac{\{}{\}}{0pt}{}{k}{\ell}$ denote the Stirling number of the second kind, i.e., the number of partitions of a $k$-element set into exactly $\ell$ nonempty blocks. Grouping partitions by $\ell=|\pi|$, we have
$$
\sum_\pi\mathbb P(\Pi=\pi)
=\sum_{\ell=1}^k \genfrac{\{}{\}}{0pt}{}{k}{\ell} t^{k-\ell}\prod_{j=1}^{\ell-1}(1-jt).
$$
Put $x=\frac{1}{t}$. Since $1-jt=\frac{x-j}{x}$ and $t^{k-\ell}\cdot t^{\ell-1}=t^{k-1}$, we get
$$
\sum_\pi\mathbb P(\Pi=\pi)
=t^{k-1}\sum_{\ell=1}^k \genfrac{\{}{\}}{0pt}{}{k}{\ell}(x-1)_{\ell-1},
$$
where $(y)_{a}=y(y-1)\cdots(y-a+1)$. We use the standard Stirling identity 
$$
 x^k=\sum_{\ell=1}^k\genfrac{\{}{\}}{0pt}{}{k}{\ell} (x)_{\ell}=x\sum_{\ell=1}^k\genfrac{\{}{\}}{0pt}{}{k}{\ell}(x-1)_{\ell-1}.
$$
 Hence, 
$$
\sum_\pi\mathbb P(\Pi=\pi)=t^{k-1}x^{k-1}=t^{k-1}\cdot t^{-(k-1)}=1.
$$

Now fix a nonempty subset $D\subseteq S$, and let $d=|D|$ and $a=k-d$. If $D$ appears as a block of $\Pi$, the remaining $a$ elements of $S\setminus D$ are partitioned into, say, $h$ nonempty blocks. Summing over all possible partitions of $S\setminus D$, we obtain
$$
\mathbb P(D\in\Pi)
=\sum_{h=0}^a \genfrac{\{}{\}}{0pt}{}{a}{h}
t^{k-(h+1)}\prod_{j=1}^{h}(1-jt)
=t^{d-1}\sum_{h=0}^a \genfrac{\{}{\}}{0pt}{}{a}{h}
t^{a-h}\prod_{j=1}^{h}(1-jt),
$$
where the term $h=0$ corresponds to the case $S\setminus D=\emptyset$, in which the product is empty and equals $1$. 
Again using $x=\frac{1}{t}$, we have $t^{a-h}\prod_{j=1}^{h}(1-jt)=t^a(x-1)_h$. Thus,
\[
\mathbb P(D\in\Pi)
=t^{d-1+a}\sum_{h=0}^a  \genfrac{\{}{\}}{0pt}{}{a}{h}(x-1)_h
=t^{d-1+a}(x-1)^a=t^{d-1}(1-t)^{k-d},
\]
where the second equality from the  Stirling identity.
This completes the proof.
\end{proof}

\begin{prop}[Supercritical slice bound]\label{prop:slicebound}
In the setting of this section, we have
$$
 \mu(\mathcal{A})\le p-p\left(1-(r-1)\frac{q}{p}\right)(1-w_H).
$$
In particular $\mu(\mathcal{A})\le p$.  Moreover, equality $\mu(\mathcal{A})=p$ forces
$w_H=1$, equivalently $\mathcal{A}_H=2^T$.
\end{prop}

\begin{proof}
Fix a coordinate $1\in H$ and put $U=H\setminus\{1\}$. Since  $|H|=m\le r$, we have $|U|=m-1\le r-1=s$.
For $B,D\subseteq U$, define
$$
 x_B=w_B,
 \qquad y_D=1-w_{H\setminus D}.
$$
We first prove a local inequality.  Fix $B\subseteq U$.  If
$D_1,\ldots,D_s\subseteq B$ and $D_1\cup\cdots\cup D_s=B$, then
$$
 B\cap(H\setminus D_1)\cap\cdots\cap(H\setminus D_s)=\emptyset.
$$
By Lemma \ref{lem:slice},
\begin{equation}\label{eq:local-cover}
 x_B=w_B\le\sum_{j=1}^s\bigl(1-w_{H\setminus D_j}\bigr)
        =\sum_{j=1}^s y_{D_j}.
\end{equation}

For $B\ne\emptyset$, choose the random partition $\Pi$ of $B$ from
Lemma \ref{lem:partition} and add $s-|\Pi|$ empty blocks.  Taking expectations in
\eqref{eq:local-cover}, and using
$$
\sum_{j=1}^s y_{D_j}
=
\sum_{\emptyset\neq D\subseteq B,\ D\in\Pi} y_D
+
(s-|\Pi|)\,y_\emptyset,\quad \mathbb E[|\Pi|]=\sum_{\emptyset\ne D\subseteq B}\mathbb P(D\in\Pi)
          =\frac{1-(1-t)^{|B|}}{t},
$$
we obtain
\begin{align}
 x_B&\le
\mathbb E\left[\sum_{\emptyset\neq D\subseteq B}\mathbf 1_{\{D\in\Pi\}}\,y_D\right]
+
\mathbb E[s-|\Pi|]\,y_\emptyset\notag\\
 &=\sum_{\emptyset\ne D\subseteq B}t^{|D|-1}(1-t)^{|B|-|D|}y_D
     +\left(s-\frac{1-(1-t)^{|B|}}{t}\right)y_\emptyset \notag\\
 &=\sum_{D\subseteq B}t^{-1}t^{|D|}(1-t)^{|B|-|D|}y_D
     -\frac{1-st}{t}y_\emptyset .        \label{eq:local-ineq}
\end{align}

For $B=\emptyset$, take $D_1=\cdots=D_s=\emptyset$. These sets cover $B$. Applying the local inequality \eqref{eq:local-cover} to this cover gives
$
x_\emptyset \le s y_\emptyset,
$
which is precisely \eqref{eq:local-ineq} in the case $B=\emptyset$

Multiply \eqref{eq:local-ineq} by $t p^{|B|}q^{|U|-|B|}$ and sum over
$B\subseteq U$. The left-hand side becomes
$$
t\sum_{B\subseteq U}p^{|B|}q^{|U|-|B|}x_B.
$$
The right-hand side becomes
\begin{align*}
&\sum_{B\subseteq U}\sum_{D\subseteq B}
p^{|B|}q^{|U|-|B|}t^{|D|}(1-t)^{|B|-|D|}y_D-(1-st)y_\emptyset\\
=&\sum_{D\subseteq U}y_D
\sum_{B\supseteq D}
p^{|B|}q^{|U|-|B|}t^{|D|}(1-t)^{|B|-|D|}-(1-st)y_\emptyset.
\end{align*}
 For fixed $D\subseteq U$, write $B=D\cup E$ with $E\subseteq U\setminus D$. Then
\begin{align*}
 &\sum_{B\supseteq D}
p^{|B|}q^{|U|-|B|}t^{|D|}(1-t)^{|B|-|D|}
=
p^{|D|}t^{|D|}
\sum_{E\subseteq U\setminus D}
p^{|E|}q^{|U|-|D|-|E|}(1-t)^{|E|}\\
=&
p^{|D|}t^{|D|}
\bigl(p(1-t)+q\bigr)^{|U|-|D|}
=
p^{|D|}t^{|D|}
p^{|U|-|D|}
=
q^{|D|}p^{|U|-|D|},
\end{align*}
where we use the binomial theorem, $p(1-t)+q=p$, and $pt=q$.
Therefore, 
\begin{equation}\label{eq:weighted-local}
 t\sum_{B\subseteq U}p^{|B|}q^{|U|-|B|}x_B
 \le \sum_{D\subseteq U}q^{|D|}p^{|U|-|D|}y_D-(1-st)y_\emptyset.
\end{equation}

Now compare $\mathcal{A}$ with the star centred at coordinate $1$.  Using
\eqref{eq:slice-measure}, we obtain
\begin{align*}
 \mu(\mathcal{A})-p
 &=q\sum_{B\subseteq U}p^{|B|}q^{|U|-|B|}x_B
   -p\sum_{D\subseteq U}q^{|D|}p^{|U|-|D|}y_D \\
 &=p\left(
 t\sum_{B\subseteq U}p^{|B|}q^{|U|-|B|}x_B
 -\sum_{D\subseteq U}q^{|D|}p^{|U|-|D|}y_D\right) \\
 &\le -p(1-st)y_\emptyset.
\end{align*}
Since $y_\emptyset=1-w_H$ and $pst=(r-1)q$, the inequality follows.
Observe that $1-st>0$.
If $\mu(\mathcal{A})=p$, then $w_H=1$. Thus, $\mathcal{A}_H$ has tail measure one. Since every tail
atom has positive measure, this is equivalent to $\mathcal{A}_H=2^T$.
\end{proof}

\section{Proof of Theorem \ref{M1}}\label{se4}

The previous proposition gives an upper bound and also identifies a necessary
condition for equality.  We now show that this condition forces the original
family to be contained in a star.

\begin{lemma}[Collapse lemma]
\label{lem:collapse}
Assume the setting of Proposition  \ref{prop:slicebound}.  If $\mu(\mathcal{A})=p$, then there
exists a coordinate $j\in H$ such that
$$
 \mathcal{A}\subseteq\mathcal{S}_j.
$$
\end{lemma}
\begin{proof}
By Proposition \ref{prop:slicebound}, equality implies $w_H=1$.  Hence,
$\mathcal{A}_H=2^T$, and in particular the set $H$ itself belongs to $\mathcal{A}$.
Let
$$
 \mathcal{W}=\{B\subseteq H:\mathcal{A}_B\ne\emptyset\}.
$$
Since $\mathcal{A}$ is increasing, $\mathcal{W}$ is an upset in $2^H$.  Since
$H\in\mathcal{A}$, any $r-1$ members of $\mathcal{A}$, together with $H$, must have a
common point in $H$.  Therefore, $\mathcal{W}$ is $(r-1)$-wise intersecting:
\begin{equation}\label{eq:W-rminus1-wise}
 B_1\cap\cdots\cap B_{r-1}\ne\emptyset
 \qquad\text{for all }B_1,\ldots,B_{r-1}\in\mathcal{W}.
\end{equation}

We claim that $\mathcal{W}$ has a common point.  
When $|H|\le r-1$, this is immediate.  Indeed, if $\bigcap_{B\in\mathcal{W}}B=\emptyset$, then for each point of $H$, we can
choose one member of $\mathcal{W}$ missing it. These at most $r-1$ members would have
empty intersection, contradicting \eqref{eq:W-rminus1-wise}.

It remains to exclude the case $|H|=r$ and
$\bigcap_{B\in\mathcal{W}}B=\emptyset$.  For each $h\in H$, choose
$B_h\in\mathcal{W}$ with $h\notin B_h$.  If $B_h$ also missed some $k\ne h$, then
the $r-1$ sets consisting of $B_h$ and $B_\ell$ for all
$\ell\in H\setminus\{h,k\}$ would have empty intersection, contradicting
\eqref{eq:W-rminus1-wise}.  Hence, $B_h=H\setminus\{h\}$ for every $h\in H$.
The same argument shows that no member of $\mathcal W$ can miss two points, since such a set together with the appropriate $H\setminus\{h\}$'s would form $r-1$ sets with empty intersection, contradicting \eqref{eq:W-rminus1-wise}.
Since
$\mathcal{W}$ is an upset, we get
\begin{equation}\label{eq:threshold-support}
\mathcal{W}=\{H\}\cup\{H\setminus\{h\}:h\in H\}.
\end{equation}

For $h\in H$, the tail families $\mathcal{A}_{H\setminus\{h\}}$ are
$r$-cross intersecting, because the high parts $H\setminus\{h\}$ have empty intersection when $h$ ranges over $H$.  By Lemma \ref{lem:additive}, we have
$$
 \sum_{h\in H}w_{H\setminus\{h\}}\le r-1.
$$
Using \eqref{eq:threshold-support} and $w_H=1$, we obtain
$$
 \mu(\mathcal{A})
 \le p^r+p^{r-1}q\sum_{h\in H}w_{H\setminus\{h\}}
 \le p^r+(r-1)p^{r-1}q.
$$
For $\frac{r-1}{r}=p_*<p<1$, we have
$$
 p^r+(r-1)p^{r-1}q=p^{r-1}\bigl((r-1)-(r-2)p\bigr)<p.
$$
Indeed, $p^{r-2}((r-1)-(r-2)p)$ is strictly increasing on $(0,1)$ and equals
$1$ at $p=1$.  This contradicts $\mu(\mathcal{A})=p$.

Therefore, $\bigcap_{B\in\mathcal{W}}B\ne\emptyset$.  Choose
$j\in\bigcap_{B\in\mathcal{W}}B$.  Every set in $\mathcal{A}$ has high part in $\mathcal{W}$, and
hence contains $j$.  Thus, $\mathcal{A}\subseteq\mathcal{S}_j$.
\end{proof}

With this collapse lemma in hand, we are ready to prove the main theorem.

\begin{proof}[Proof of Theorem \ref{M1}]
The stars $\mathcal S_i$ with $p_i=p_1$ are $r$-wise intersecting and satisfy
$\mu_{\mathbf p}(\mathcal S_i)=p_1$.
Let $p_*=\frac{r-1}{r}$.
We first prove the theorem for increasing families.  Let $\mathcal A$ be increasing
and $r$-wise intersecting.

\smallskip
\noindent\textit{Case 1: $p_1<p_*$.}
Then $p_2<p_*$. Applying Theorem \ref{T22} directly gives both
$\mu_{\mathbf p}(\mathcal A)\le p_1$ and the stated equality classification.

\smallskip
\noindent\textit{Case 2: $p_1=p_*$.}
All coordinates have probability at most $p_*$.  Applying
Lemma~\ref{lem:additive} to the $r$ identical families
$\mathcal A,\ldots,\mathcal A$ gives
$$
 r\mu_{\mathbf p}(\mathcal A)\le r-1.
$$
Hence, $\mu_{\mathbf p}(\mathcal A)\le p_*=p_1$.

Assume equality holds.  Let $j$ be any coordinate with $p_j<p_*$.  Raise
only $p_j$ to $p_*$ and leave all other probabilities fixed.  Since all
coordinates are still at most $p_*$, the same additive bound gives measure
at most $p_*$.  Monotonicity and the equality
$\mu_{\mathbf p}(\mathcal A)=p_*$ imply that the measure did not increase.  By
the strict monotonicity part of Lemma \ref{lem:monotone}, $\mathcal A$ is independent of
coordinate $j$.  Repeating this for all $j$ with $p_j<p_*$, we find that
$\mathcal A$ depends only on
$$
 M=\{j:p_j=p_*\}.
$$
The assumption $p_{r+1}<p_*$ gives $|M|\le r$.

As a family on the ground set $M$, $\mathcal A$ must have a common point.  Otherwise,
for each point of $M$, we could choose a member of $\mathcal A$ missing it. Since
$|M|\le r$, after repetitions this would give $r$ members of $\mathcal A$ with
empty intersection.  Thus $\mathcal A\subseteq\mathcal S_j$ for some $j\in M$.  Since
$\mu_{\mathbf p}(\mathcal A)=p_*=\mu_{\mathbf p}(\mathcal S_j)$ and all atoms have
positive measure, we have $\mathcal A=\mathcal S_j$.

\smallskip
\noindent\textit{Case 3: $p_1>p_*$.}
Let
\[
 H=\{j:p_j>p_*\}.
\]
Then $1\le |H|\le r$.  Define a new probability vector $\mathbf p^+$ by raising
every coordinate in $H$ to $p_1$ and leaving the tail coordinates unchanged.
Since $\mathcal A$ is increasing, Lemma \ref{lem:monotone} gives
$$
 \mu_{\mathbf p}(\mathcal A)\le \mu_{\mathbf p^+}(\mathcal A).
$$
The high coordinates of $\mathbf p^+$ have common probability $p_1$, while all
tail probabilities are at most $p_*$.  Hence, Proposition \ref{prop:slicebound} gives
$\mu_{\mathbf p^+}(\mathcal A)\le p_1$.  Thus, $\mu_{\mathbf p}(\mathcal A)\le p_1$.

If equality holds for the original vector $\mathbf p$, then
$\mu_{\mathbf p^+}(\mathcal A)=p_1$.  By Lemma \ref{lem:collapse}, there exists
$j\in H$ such that $\mathcal A\subseteq\mathcal S_j$.  Returning to the original vector,
\[
 p_1=\mu_{\mathbf p}(\mathcal A)\le\mu_{\mathbf p}(\mathcal S_j)=p_j\le p_1.
\]
Therefore, $p_j=p_1$ and 
$\mathcal A=\mathcal S_j$.

This completes the proof for increasing families.

Now let $\mathcal A$ be an arbitrary $r$-wise intersecting family.  By
Lemma~\ref{lem:monotone}, its upset $\mathcal A^\uparrow$ is increasing and
$r$-wise intersecting, and
$$
 \mu_{\mathbf p}(\mathcal A)\le\mu_{\mathbf p}(\mathcal A^\uparrow)\le p_1.
$$
If equality holds, then $\mu_{\mathbf p}(\mathcal A^\uparrow)=p_1$.  The increasing
case gives $\mathcal A^\uparrow=\mathcal S_i$ for some $i$ with $p_i=p_1$.  Since
$\mathcal A\subseteq\mathcal S_i$ and all atoms of $\mathcal S_i$ have positive measure, the
equality $\mu_{\mathbf p}(\mathcal A)=\mu_{\mathbf p}(\mathcal S_i)$ implies
$\mathcal A=\mathcal S_i$.
\end{proof}

\section*{Declaration of competing interest}
We declare that we have no conflict of interest to this work.

\section*{Data availability}
No data was used for the research described in the article.


\end{document}